\documentclass{birkjour}
\usepackage{hyperref}
\newtheorem{thm}{Theorem}[section]
\newtheorem{cor}[thm]{Corollary}
\newtheorem{lem}[thm]{Lemma}

\theoremstyle{definition}
\newtheorem{defn}[thm]{Definition}
\theoremstyle{remark}
\newtheorem{rem}[thm]{Remark}
\newtheorem{ex}{Example}
\numberwithin{equation}{section}

\newcounter{tmp}

\begin{document}
	
	%
	%
	%
	%
	%
	%
	%
	%
	%

	\title[Twisted Derivations of Group Rings]
	{The Twisted Derivation Problem for Group Rings}

	\author[D. Chaudhuri]{Dishari Chaudhuri}
	
	\address{
		Department of Mathematical Sciences\\ Indian Institute of Science Education and Research Mohali\\
		Sector-81, Knowledge City, S.A.S. Nagar, Mohali-140306\\ Punjab, India}
	
	\email{dishari@iitg.ac.in, dishari.chaudhuri@gmail.com}
	
	\thanks{The author is thankful to IISER Mohali for providing fellowship when this project was carried out and the unknown referee whose valuable comments improved the article to a great extent.}
	
	
	
	
	\begin{abstract}We study $(\sigma,\tau)$-derivations of a group ring $RG$ where $G$ is a group with center having finite index in $G$ and $R$ is a semiprime ring with $1$ such that either $R$ has no torsion elements or that if $R$ has $p$-torsion elements, then $p$ does not divide the order of $G$ and let $\sigma,\tau$ be $R$-linear endomorphisms of $RG$
		fixing the center of $RG$ pointwise. We generalize Main Theorem $1.1$ of \cite{Chau-19} and prove that there is a ring $T\supset R$ such that $\mathcal{Z}(T)\supset\mathcal{Z}(R)$ and that for the natural extensions of $\sigma, \tau$ to $TG$ we get $H^1(TG,{}_\sigma TG_\tau)=0$, where ${}_\sigma TG_\tau$ is the twisted $TG-TG$-bimodule. We provide applications of the above result and Main Theorem $1.1$ of \cite{Chau-19} to integral group rings of finite groups and connect twisted derivations of integral group rings to other important problems in the field such as the Isomorphism Problem and the Zassenhaus Conjectures. We also give an example of a group $G$ which is both locally finite and nilpotent and such that for every field $F$, there exists an $F$-linear $\sigma$-derivation of $FG$ which is not $\sigma$-inner.
	
	\end{abstract}
	
	\maketitle


%

\maketitle

\section{Introduction}
Let $R$ be a commutative ring with $1$. Let $\mathcal{A}$ be an $R$-algebra and $N$ be an $\mathcal{A}-\mathcal{A}$-bimodule. Then an $R$-linear map $\gamma: \mathcal{A}\rightarrow N$ is called a derivation if $\gamma(ab)=\gamma(a)b+a\gamma(b)$ for all $a,b\in\mathcal{A}$.  A derivation is inner if there exists $x\in N$, such that $\gamma(a)=xa-ax$ for all $a\in\mathcal{A}$. Hence a derivation is a representative of an element in the degree $1$ Hoschild cohomology $H^1(\mathcal{A},N)$ and it is inner if it represents the zero element in $H^1(\mathcal{A},N)$. Let $\sigma$, $\tau$ be two different algebra endomorphisms on $\mathcal{A}$. Then a $(\sigma,\tau)$-derivation is an $R$-linear map $\delta: \mathcal{A}\rightarrow {}_\sigma\mathcal{A}_\tau$ satisfying $\delta(ab)\;=\; \delta(a)\tau(b)+\sigma(a)\delta(b)$ for $a,b\in \mathcal{A}$. So a $(\sigma,\tau)$-derivation is the special case of a derivation with
the twisted bimodule $_\sigma \mathcal{A}_\tau$ in the role of $N$. Here $_\sigma \mathcal{A}_\tau$ is  $\mathcal{A}$ as $R$-space and on it $a$ acts on the left as multiplication by $\sigma(a)$ and on the right by multiplication by $\tau(a)$. If there exists $x\in\mathcal{A}$ such that the $(\sigma,\tau)$-derivation $\delta_x:\mathcal{A}\rightarrow{}_\sigma\mathcal{A}_\tau$ is of the form $\delta_{x}(a)=x\tau(a)-\sigma(a)x$, then $\delta_x$ is called a $(\sigma,\tau)$-inner derivation of $\mathcal{A}$ induced by $x$. If $\sigma=\tau=id$, then $\delta$ and $\delta_x$ are respectively the usual derivation and inner derivation of $\mathcal{A}$ induced by $x$.\\

 Such twisted derivations were mentioned by Jacobson as $(s_1,s_2)$-derivations in \cite{J} (Chapter $7.7$) and were used in the study of generalization of Galois theory over division rings. Later on they have been more commonly referred to as $(\sigma,\tau)$-derivation or $(\alpha,\beta)$-derivation or $(\theta,\phi)$-derivation (e.g., Arga\c{c} et al. \cite{AKK-1987}, Bre\v{s}ar and Vukman \cite{BV-1991} to name just a few) and have been studied extensively for prime and semiprime rings (see \cite{AAH-06} for a survey on such results). They have also been used in the study of $q$-difference operators in number theory (\cite{Andre-2001}, \cite{vizio-2002}). In the last decade some major applications of $(\sigma,\tau)$-derivations were given by Hartwig et al. in their highly influential work \cite{HLS-06}. They found a way for the study of deformations of Witt algebra and constructed generalizations of Lie algebras known as hom-Lie algebras with the help of such derivations. Just as algebras of derivations become Lie algebras, hom-Lie algebras were constructed so that they appear as algebras of twisted derivations wth some added natural conditions. Since then the study of these kinds of twisted derivations have gained a new momentum as hom-Lie algebras form very interesting mathematical objects and are used in the study of deformations and discretizations of vector fields that have widespread applications to quantum physics, algebraic geometry and number theory. For an elaborate motivation and philosophy behind the construction of such structures with the help of twisted derivations one might refer to \cite{Lar-2017}.\par
 Twisted derivations were introduced and studied mostly in the commuative ring setting. The most common examples of $(\sigma,\tau)$-derivations are of the form $a(\tau-\sigma)$ for some suitable element $a$ in the ring or some extension of it. It can be shown in many commutative rings \emph{all} twisted derivations are of this form. Exclusive examples can be found in Tables $1$ of \cite{H-02} and \cite{ELMS-16}. We however are interested in twisted derivations of non-commutative rings, in particular, group rings. As we will see in many group rings also all twisted derivations can be of similar form. But we also provide an example when it is not of the above form. Ordinary derivations on group rings are well studied by Spiegel \cite{Sp-94} and Ferrero et al. \cite{FGM-95}. We will denote by $\mathcal{Z}(\mathcal{A})$ the center of the algebraic object (group or ring) $\mathcal{A}$. In an earlier work, we have proved the following:
 \begingroup
 \setcounter{tmp}{\value{thm}}
 \setcounter{thm}{0} 
 \renewcommand\thethm{\Alph{thm}}
  \begin{thm}[Main Theorem $1.1$, \cite{Chau-19}]\label{previous}
 	If $G$ is a finite group and if $R$ is an integral domain with $1$ such that either $R$ has no torsion elements or that if $R$ has $p$-torsion elements, then $p$ does not divide the order of $G$, then:
 	\begin{enumerate}
 		\item \label{part 1}if $R$ is a field and $\sigma$, $\tau$ are central $R$-algebra endomorphisms of $RG$, that is, $\sigma$ and $\tau$ fix $\mathcal{Z}(RG)$ elementwise, then $H^1(RG,{}_\sigma RG_\tau)=0$.
 		\item\label{part 2} if $R$ is an integral domain which is not a field and $\sigma,\;\tau$ are $R$-linear extensions of group homomorphisms of $G$ such that they fix $\mathcal{Z}(RG)$ elementwise, then $H^1(RG,{}_\sigma RG_\tau)=0$.
 	\end{enumerate}
 \end{thm}
\endgroup
   Now the definition of a $(\sigma,\tau)$-derivation is given for only one pair of endomorphisms $\sigma$, $\tau$ at a time and different endomorphisms might give different properties for such derivations. Thus it is essential to study $(\sigma,\tau)$-derivations of an algebra for different endomorphisms on it. Also $\sigma$ and $\tau$ are taken to be $R$-linear extensions of group endomorphisms of $G$ in Theorem \ref{previous}, part \ref{part 2}. Due to such strong assumptions on $\sigma$ and $\tau$ many interesting cases like changes of group bases inside the group ring, e.g., by conjugation with units, are excluded. Also $G$ was assumed to be finite in Theorem \ref{previous}. In this present work, we take into account most of such cases and examine the properties in much more generality as we take $R$ to be a semiprime ring with $1$, $G$ a torsion group with its center having finite index in $G$ and $\sigma$, $\tau$ to be $R$-linear central endomorphisms of $RG$, that is, they fix the center of $RG$ elementwise. So here $\sigma(G),\tau(G)\subseteq RG$. Then we apply our results to integral group rings of finite groups and connect twisted derivations of group rings to other important problems of the field such as the Isomorphism Problem or the Zassenhaus Conjectures. 
 
  Recall that a ring $\mathcal{R}$ is said to be of $characteristic\;0$ if $\mathcal{R}$ does not have torsion elements. Otherwise there exists a set of prime integers $p$ for which there is a $p$-torsion element.  We now state our main result:

\setcounter{thm}{\thetmp}
\begin{thm}\label{main}
	
	Let $G$ be a torsion group whose center has finite index. Let $R$ be semiprime with $1$
	such that either $R$ does not have torsion elements or that if $R$ has $p$-torsion elements, then $p$ does not divide the order of $G$ and let $\sigma,\tau$ be $R$-linear endomorphisms of $RG$
	fixing the center of $RG$ pointwise. Then there is a ring $T$ containing $R$ such that the center of $T$ contains the center of $R$ and such that for the natural extensions of $\sigma, \tau$ to $TG$ we get $H^1(TG,{}_\sigma TG_\tau)=0$.
	
\end{thm}

Our manuscript has been divided into four sections. Section \ref{Section 2} contains some known results on $(\sigma,\tau)$-derivations to be used later and a brief overview of the important problems on integral group rings. Section \ref{Section 3} is devoted to the proof of Theorem \ref{main}. In section \ref{Section 4} we apply Theorem \ref{main} and Theorem \ref{previous} to integral group rings of finite groups (Theorems \ref{nec suf inner}, \ref{application}) and also comment on twisted derivations of commutative group algebras (Theorem \ref{commutative group algebra}). We also provide a counter example (\ref{Counter Example}) to show that Theorem \ref{main} may not hold in general.




\section{Useful Results}\label{Section 2}

 Let $R$ be a commutative ring with $1$ and $\mathcal{A}$ be an algebra over $R$. Let $\sigma$, $\tau$ be two different $R$-algebra endomorphisms on $\mathcal{A}$. The set of all $R$-linear $(\sigma,\tau)$-derivations on $\mathcal{A}$ will be denoted by $\mathfrak{D}_{(\sigma,\tau)}(\mathcal{A})$. The following properties of $(
 \sigma,\tau)$-derivations can be checked very easily and can also be found in Section $2$ of \cite{Chau-19}. All such derivations are assumed to be $R$-linear in the following.

\begin{lem}\label{useful lemma}The following properties are satisfied by $(\sigma,\tau)$-derivations on $\mathcal{A}.$ \begin{enumerate}
		\item If $\mathcal{A}$ is unital, then for any $(\sigma,\tau)$-derivation $\delta$, $\delta(1)=0$.
		\item $\mathfrak{D}_{(\sigma,\tau)}(\mathcal{A})$ is an $R$-module as $\delta_1+\delta_2,\;r\delta_1\in\mathfrak{D}_{(\sigma,\tau)}(\mathcal{A})$ for $\delta_1,\delta_2\in\mathfrak{D}_{(\sigma,\tau)}(\mathcal{A})$ and $r\in R$.
		\item When $\sigma(x)a=a\sigma(x)$ $(\text{or } \tau(x)a=a\tau(x))$ for all $x,a\in\mathcal{A}$, and in particular when $\mathcal{A}$ is commutative, $\mathfrak{D}_{(\sigma,\tau)}(\mathcal{A})$ carries a natural left (or right) $\mathcal{A}$-module structure by $(a,\delta)\longmapsto a.\delta:\;x\mapsto a\delta(x)$.
		\item \label{inner sum}For $x,y\in\mathcal{A}$, the $(\sigma,\tau)$-inner derivations satisfy: $\delta_{x+y}=\delta_x+\delta_y$.
		\item \label{sigma inner commutator}For $(\sigma,\tau)$-inner derivations $\delta_x, \delta_y$ for some $x,y\in\mathcal{A}$, $\delta_x=\delta_y$ if and only if $(x-y)\tau(a)=\sigma(a)(x-y)$ for all $a\in\mathcal{A}$.
		\item \label{induction}
		Let $\sigma$ and $\tau$ be algebra homomorphisms on $\mathcal{A}$ that fix $\mathcal{Z}(\mathcal{A})$ elementwise. Then for a $(\sigma,\tau)$- derivation $\delta$ on $\mathcal{A}$, we have $\delta\left(\alpha^n\right)=n\alpha^{n-1}\delta(\alpha)$ for all $\alpha\in\mathcal{Z}(\mathcal{A})$.
	\end{enumerate}
\end{lem}

The following is Corollary $2.9$ of \cite{Chau-19}.
\begin{cor}\label{csa-sigma-inner} Let $\mathcal{A}$ be a finite dimensional central simple algebra with $1$ over a field $F$. Let $\sigma$ and $\tau$ be non-zero $F$-algebra endomorphisms of $\mathcal{A}$. Then $H^1(\mathcal{A},{}_\sigma\mathcal{A}_\tau)=0$.
\end{cor}
The following is Theorem $4$ of \cite{HLS-06}.
\begin{thm}\label{ufd thm}If $\sigma$ and $\tau$ are two different algebra endomorphisms on a unique factorization domain $\mathcal{A}$, then $\mathcal{D}_{(\sigma,\tau)}(\mathcal{A})$ is free of rank one as an $\mathcal{A}$-module with generator $$\Delta:=\frac{\tau-\sigma}{g}:\;x\longmapsto\frac{(\tau-\sigma)(x)}{g},\text{ where }g=\;gcd\;((\tau-\sigma)(\mathcal{A})).$$
\end{thm}

We now prove a useful result to be used in the proof of our main theorem in Section $3$. 
\begin{lem} \label{extension}Let $R$ and $T$ be rings with the same identity such that $R\subset T$ and $\mathcal{Z}(R)\subset\mathcal{Z}(T)$ and $G$ be a group. Let $\sigma$ and $\tau$ be central $R$-endomorphisms of $RG$ and $\delta$ be an $R$-linear $(\sigma,\tau)$-derivation of $RG$. Then $\delta$ can be extended to a $T$-linear $(\sigma_T,\tau_T)$-derivation $\delta_T$ of $TG$, where $\sigma_T, \tau_T$ are natural $T$-linear extensions of $\sigma$ and $\tau$ to $TG$ that fix $\mathcal{Z}(TG)$ elementwise. 
\end{lem}
\begin{proof}
We first observe that if $Z=\mathcal{Z}({R})$ is the center of $R$, then given $r\in R$, $g\in G$, we have $rg=gr$. So $\delta(rg)=\delta(gr)$, that is, $r\delta(g)=\delta(g)r$. Hence $\delta(g)\in ZG$. In the same way, we can verify that $\sigma(g),\tau(g)\in ZG$ for all $g\in G.$ Thus, $\delta, \sigma$ and $\tau$ restricted to $ZG$ induces a $Z$-linear $(\sigma,\tau)$-derivation of $ZG$. 
 Now identifying $TG$ with $T\otimes_ZZG$, $\sigma$ and $\tau$ can be extended in a natural way to $T$-linear endomorhisms $\sigma_T$ and $\tau_T$ of $TG$ by defining $\sigma_T=1\otimes\sigma$ and $\tau_T=1\otimes\tau$ respectively. Clearly, $\sigma_T$ and $\tau_T$ will fix $\mathcal{Z}(TG)$ elementwise. In fact, if $\alpha\in\mathcal{Z}(TG)$, then $\alpha=\sum_{i}t_i\otimes\alpha_i$, where $t_i\in \mathcal{Z}(T)$ and $\alpha_i\in \mathcal{Z}(ZG).$ Then, as $\sigma$ fixes $\mathcal{Z}(ZG)$ elementwise,  $\sigma_T(\alpha)=\sum_{i}t_i\otimes\sigma(\alpha_i)=\sum_it_i\otimes\alpha_i=\alpha.$ Similarly, $\tau_T(\alpha)=\alpha.$ That is, $\sigma_T$ and $\tau_T$ fix $\mathcal{Z}(TG)$ elementwise. Finally, we define $\delta_T:T\otimes_ZZG\longrightarrow T\otimes_ZZG$ by $\delta_T=1\otimes\delta$. For $g,h\in G$, we have
 \begin{eqnarray*}
	\delta_T(gh)
	&=&1\otimes\delta(gh)=1\otimes\big(\delta(g)\tau(h)+\sigma(g)\delta(h)\big)\\
	&=&\big(1\otimes\delta(g)\big)\big(1\otimes\tau(h)\big)+\big(1\otimes\sigma(g)\big)\big(1\otimes\delta(h)\big)\\
	&=&\delta_T(g)\tau_T(h)+\sigma_T(g)\delta_T(h).
\end{eqnarray*}
Since $\delta$ is $R$-linear, $\delta_T$ is $T$-linear. Hence by bilinearity, it follows that $\delta_T$ is a $T$-linear $(\sigma_T,\tau_T)$-derivation of $TG.$
\end{proof}

 Now we give a brief overview of some of the main problems in the field of integral group rings. We will denote by $\mathcal{U}(R)$ the unit group of a ring $R$. Recall that the subgroup of normalized units of the unit group of an integral group ring $\Bbb{Z}G$, that is, units of $\Bbb{Z}G$ having augmentation $1$, is denoted by $V(\Bbb{Z}G)$. We say two elements $a$ and $b$ in $\Bbb{Z}G$ are rationally conjugate if there exists $u\in\mathcal{U}(\Bbb{Q}G)$ such that $u^{-1}au=b$. Also recall the following definition:
 \begin{defn}
 	For two finite groups $G$ and $H$ an isomorphism $\phi:\Bbb{Z}G\longrightarrow\Bbb{Z}H$ is called a \textbf{normalized isomorphism} if for every element $\alpha\in \Bbb{Z}G$ we have that $\varepsilon(\alpha)=\varepsilon\big(\phi(\alpha)\big)$ (or equivalently, if for every $g\in G$ we have that $\varepsilon(\phi(g))=1$), where $\varepsilon$ is the augmentation map.
 \end{defn}
 
 We now  state the conjectures. \\
  
 \textbf{The Isomorphism Problem:}\\
 (ISO) Let $G$ and $H$ be arbitrary finite groups. Does $\Bbb{Z}G\cong\Bbb{Z}H$ imply $G\cong H?$\\

 \textbf{The Zassenhaus Conjectures, $\mathbf{1974}$:} Given a finite group $G$:\\
 (ZC$1$) Is every torsion element of $V(\Bbb{Z}G)$ rationally conjugate to an element of $G$?\\
 (ZC$2$) Is every finite subgroup of $V(\Bbb{Z}G)$, with the same order as $G$, rationally conjugate to $G$?\\
 (ZC$3$) Is every finite subgroup of $V(\Bbb{Z}G)$ rationally conjugate to a subgroup of $G$?\\
 (AUT) If $\theta$ is a normalized automorphism of $\Bbb{Z}G$, then do there exist $\beta\in Aut(G)$ and $u\in \mathcal{U}(\Bbb{Q}G)$ such that $\theta(g)=u^{-1}\beta(g)u$ for all $g\in G?$\\
 
 We have the following relations between the different conjectures. For the proofs of the following one can refer to Chapter $37$, \cite{Sehgal}. 
 
 \begin{lem}\label{conjectures}
 	\begin{enumerate}
 		\item $(ZC3)\implies\; (ZC1)$ and $(ZC2)$. 
 		\item $(ZC2)\implies(ISO)$.
 		\item $(ZC2)\implies(AUT)$.
 		\item $(AUT)+(ISO)\implies(ZC2)$.
 	\end{enumerate}
 \end{lem}

\section{Proof of Theorem \ref{main}}\label{Section 3}

 Let $RG$ be the group ring of a torsion group $G$ over a semiprime ring $R$ with $1$ such that $[G:\mathcal{Z}(G)]<\infty$. Let $\sigma$ and $\tau$ be $R$-algebra central endomorphisms on $RG$, that is, they fix the center of $RG$ elementwise. Let $\delta$ be an $R$-linear $(\sigma,\tau)$-derivation of the group ring $RG$. We need to find a ring $T$ such that $R\subset T$ and $\mathcal{Z(R)}\subset \mathcal{Z}(T)$ and some $\alpha\in TG$ such that $\delta=\delta_\alpha$. That is, we need to show that $\delta_T$ is a $T$-linear $(\sigma_T,\tau_T)$-inner derivation of $TG$ as then there will exist some $\alpha\in TG$ such that for all $g\in G$, $$\delta(g)=\delta_T(g)=\alpha\tau_T(g)-\sigma_T(g)\alpha=\alpha\tau(g)-\sigma(g)\alpha\subseteq RG,$$ 
where the notations are same as in Lemma \ref{extension}. So $\delta$ will be equal to $\delta_\alpha$ for some $\alpha\in TG.$

We first observe with the help of the arguments in the first paragraph of the proof of Lemma \ref{extension}, that $\delta, \sigma$ and $\tau$ restricted to $ZG$ induces a $Z$-linear $(\sigma,\tau)$-derivation of $ZG$. 

First assume that $G$ is finite and either $R$ does not have torsion elements or that if $R$ has $p$-torsion elements, then $p$ does not divide the order of $G$. Let $R'$ be the subring of $Z$ generated by the finitely many elements of $R$ which occur as coefficients of elements in $\delta(G),\;\sigma(G)$ and $\tau(G).$ Then $\delta$ restricts to an $R'$-linear $(\sigma,\tau)$-derivation of $R'G.$ Let us denote the restriction of $\delta,\;\sigma,\;\tau$ to $R'G$ as $\delta',\sigma',\tau'$ respectively. \\

Now, $R'$ is commutative, semiprime and Noetherian with $1$. Let $P_1,\ldots,P_n$ be the finitely many minimal prime ideals of $R'$ (as $R'$ is Noetherian) and $F_i$ be the minimal algebraically closed field containing $R'/P_i$, $1\leq i\leq n.$ That is, let $F_i$ be the algebraic closure of the field of fractions of $R'/P_i$ for each $i$. Since $R'$ is semiprime we have $\bigcap_{i=1}^{n}P_i=\{0\}.$ Hence $R'$ can be embedded in $T'=\bigoplus_{i=1}^{n}F_i.$ Now by Lemma \ref{extension}, $\delta'$ can be extended to $(\sigma'_{T'},\tau'_{T'})$-derivation $\delta'_{T'}$ of $T'G$. Thus, it is enough to prove that $\delta'_{T'}$ is $(\sigma'_{T'},\tau'_{T'})$-inner in $T'G$.\\

Note that if $F_i$ has $p_i$-torsion elements for a prime integer $p_i$, then $p_i\in P_i$. Hence, $p_i\big(\cap_{j\neq i}P_j\big)=\{0\}.$ Thus, $R'$ has $p_i$-torsion elements and therefore, $p_i$ does not divide $|G|,$ $1\leq i\leq n.$ Hence, by Maschke's Theorem, $T'G$ is the direct sum of full matrix rings over fields, say, $T'G=I_1\oplus\cdots\oplus I_k$ where each $I_j$ is generated as an ideal by a central idempotent.\\

 Let $e_i$ be the central idempotent generating $I_i$ as an ideal for each $1\leq i\leq k.$ Now as $\sigma'_{T'},\tau'_{T'}$ are central endomorphisms, for $i\neq j$ we have  $0=\delta'_{T'}(e_ie_j)=\delta'_{T'}(e_i)e_j+e_i\delta'_{T'}(e_j).$ Thus multiplying with $e_j$ we get $\delta'_{T'}(e_i)e_j=0$ for every $j\neq i$. Hence $\delta'_{T'}(e_i)\in I_i$ and this implies $\delta'_{T'}(I_i)\subseteq I_i.$ So $\delta'_{T'},\sigma'_{T'},\tau'_{T'}$ restricted to each component $I_i$ will give a $(\sigma'_{T'},\tau'_{T'})$-derivation on that component which will be inner by Corollary \ref{csa-sigma-inner} induced by an element, say, $\alpha_i\in I_i.$  Thus $\delta'_{T'}$ will be $(\sigma'_{T'},\tau'_{T'})$-inner in $T'G$ induced by the element $\alpha=\alpha_1+\cdots+\alpha_k$. So $\delta'(\gamma)=\alpha\tau'_{T'}(\gamma)-\sigma'_{T'}(\gamma)\alpha$ for all $\gamma\in T'G.$ Now we have a chain of rings, $R'\subset Z\subset R$ and $R'\subset T'$. Let $T=T'\otimes_Z R$. Then $T'\subset T$, so we have $\alpha\in TG$ and $\delta=\delta_\alpha$. Hence the result is proved for the case when $G$ is finite.\\

Now, we consider the general case when $G$ is a torsion group with $[G:\mathcal{Z}(G)]<\infty.$ We first notice that $\delta(\mathcal{Z}(G))=0.$ In fact, if $z\in\mathcal{Z}(G)$ with $o(z)=m$, then $\delta(z^m)=mz^{m-1}\delta(z)$ (by Lemma \ref{useful lemma}, part \ref{induction}) and this implies $\delta(z)=0$ as $z$ is inveritble.\\
Now let $X=\{g_1,g_2,\ldots,g_n\}$ be a transversal of $\mathcal{Z}(G)$ in $G$. For every index $i$, $1\leq i\leq n$, we write $$\delta(g_i)=\sum_{i,j,k}\alpha_{ijk}z_{ijk}g_k,\qquad z_{ijk}\in\mathcal{Z}(G),\;\alpha_{ijk}\in T.$$
Also for $i,j=1,2,\cdots,n$ let $g_ig_j=c_{ij}g_k$, $c_{ij}\in\mathcal{Z}(G).$ Denote by $H$ the subgroup of $G$ generated by all the elements $z_{ijk}, c_{ij}, g_k$. Since $G$ is torsion and $\mathcal{Z}(G)$ is abelian, it follows that $H$ is finite. Also, the restriction $\delta|_{RH}$ is an $R$-linear $(\sigma,\tau)$-derivation of $RH.$ By the first case $\delta|_{RH}={\delta_\alpha|_{RH}}$ for some $\alpha\in TH$. Now, given an element $g\in G$, write $g=zg_i$ with $z\in\mathcal{Z}(G)$, $1\leq i\leq n.$ Then:
$$\delta(g)=z\delta(g_i)=z(\alpha\tau(g_i)-\sigma(g_i)\alpha)=\alpha z\tau(g_i)-z\sigma(g_i)\alpha=\alpha\tau(g)-\sigma(g)\alpha.$$ 
Hence, the result.

\section{Applications}\label{Section 4}

We first discuss the necessary and sufficient conditions for a $\Bbb{Z}$-linear $(\sigma,\tau)$-derivation on $\Bbb{Z}G$ to be $(\sigma,\tau)$-inner for a given pair of endomorphisms $\sigma$ and $\tau$ on $\Bbb{Z}G.$ Recall that the notation $a|b$ means $a$ divides $b$ for $a,b\in \Bbb{Z}$.

\begin{thm}\label{nec suf inner} Let $G$ be a finite group and $\sigma$ and $\tau$ be central endomorphisms of $\Bbb{Z}G$. Let $\delta$ be a $\Bbb{Z}$-linear $(\sigma,\tau)$-derivation of $\Bbb{Z}G$. Let us denote $\delta(g)=\sum_{x\in G}m^g_xx,\;\tau(g)=\sum_{t\in G}c^g_tt$ and $\sigma(g)=\sum_{s\in G}b^g_ss$ for every $g\in G$. Then $\delta$ will be $(\sigma,\tau)$-inner in $\Bbb{Z}G$ if and only if ${gcd}_{h\in G}\big(c^g_{h^{-1}x}-b^g_{xh^{-1}}\big)\Big|m^g_x$ for every $g\in G$ and $x\in G$. \end{thm}

\begin{proof}As $\sigma_\Bbb{Q},\;\tau_\Bbb{Q}$ are central $\Bbb{Q}$-automorphisms of $\Bbb{Q}G$, by Theorem \ref{previous}, part \ref{part 1}, $\delta_\Bbb{Q}$ will be $(\sigma_\Bbb{Q},\tau_\Bbb{Q})$-inner in $\Bbb{Q}G$. Hence, $\delta(g)=\alpha\tau(g)-\sigma(g)\alpha\subseteq\Bbb{Z}G$ for all $g\in G$ and some $\alpha\in\Bbb{Q}G$. Now,
	\begin{eqnarray*}\delta(g)&=&\alpha\tau(g)-\sigma(g)\alpha\\
		&=&\sum_{h\in G}\alpha_hh\sum_{t\in G}c^g_tt- \sum_{s\in G}b^g_ss\sum_{h\in G}\alpha_hh\\
		&=&\sum_{ht=x}\alpha_hc^g_tx-\sum_{sh=x}b^g_s\alpha_hx=\sum_{x\in G}m^g_xx\in\Bbb{Z}G.
	\end{eqnarray*}
	So we have for every $x\in G$, $$\sum_{h\in G}\alpha_h\big(c^g_{h^{-1}x}-b^g_{xh^{-1}}\big)=m^g_x.$$
	Thus for each $x\in G$ we have a non-homogeneous linear equation in $n$ variables $x_1,x_2,\ldots,x_n$ with integral coefficients where $n$ is the order of $G$. That is, we have:
	$$\Big(c^g_{{h_1^{-1}}x}-b^g_{x{h_1^{-1}}}\Big)x_1+\Big(c^g_{{h_2^{-1}}x}-b^g_{x{h_2^{-1}}}\Big)x_2+\cdots+\Big(c^g_{{h_n^{-1}}x}-b^g_{x{h_n^{-1}}}\Big)x_n=m^g_x,$$
	where $h_i\in G$, $1\leq i\leq n$. Again for each $g\in G$ we will have another $n$ such equations and so on. Now $\alpha_{h_i}\in\Bbb{Q}$, $1\leq i\leq n$ is a solution for all these equations. As the coefficients are all integers, these equations will have an integer solution if and only if ${gcd}_{h\in G}\big(c^g_{h^{-1}x}-b^g_{xh^{-1}}\big)\Big|m^g_x$ for every $g\in G$. This in turn implies that $\delta$ will be $(\sigma,\tau)$-inner in $\Bbb{Z}G$ if and only if ${gcd}_{h\in G}\big(c^g_{h^{-1}x}-b^g_{xh^{-1}}\big)\Big|m^g_x$ for every $g\in G$.
\end{proof}

\begin{cor}
	Let $G$ be a finite group and $\sigma$ and $\tau$ be group endomorphisms that can be realized by conjugation with a unit in $\Bbb{Z}G$ and such that they are central endomorphisms of $\Bbb{Z}G$. Then $H^1(\Bbb{Z}G,{}_\sigma\Bbb{Z}G_\tau)=0$.
\end{cor}
\begin{proof}
	It follows easily from Theorem \ref{previous} part \ref{part 2} and Theorem \ref{nec suf inner}.
\end{proof}

\begin{rem}
	Note that now we obtain Theorem \ref{previous} part \ref{part 2} for the case $R=\Bbb{Z}$ as a corollary of Theorem \ref{nec suf inner} as $\sigma(G),\tau(G)\subseteq G$. The elements $\big(c^g_{h^{-1}x}-b^g_{xh^{-1}}\big)$ will be either $0$ or $1$. Hence the result. 
\end{rem}

Now we discuss the connection of the main conjectures to the twisted derivations of integral group rings. For a ring $R$, we will denote by $[R,R]$ the additive subgroup of $R$ generated by the Lie brackets $[x,y]=xy-yx$ for $x,y\in R.$ Given any automorphism $\phi$ of $\Bbb{Z}G$, we can define a unique normalized automorphism $\psi$ of $\Bbb{Z}G$ in the following manner: 
$$\psi\Big(\sum_{g\in G}r_gg\Big)=\sum_{g\in G}r_g\varepsilon(\phi(g))^{-1}\phi(g).$$
So very little generality is lost if we assume $\sigma$ and $\tau$ to be normalized automorphisms of $\Bbb{Z}G$ in the next theorem.
\begin{thm}\label{application}
	Let $G$ be a finite group and $\sigma$ and $\tau$ be normalized central automorphisms of $\Bbb{Z}G$. If $\Bbb{Z}G$ has a positive solution to $(ZC2)$, that is, every finite subgroup of $V(\Bbb{Z}G)$ with the same order as $G$ is rationally conjugate to $G$, then:
	\begin{enumerate}
		\item \label{Appl 1}There will exist $u\in \mathcal{U}(\Bbb{Q}G)$ and  $\alpha\in\Bbb{Q}G$ such that any $\Bbb{Z}$-linear $(\sigma,\tau)$-derivation $\delta$ of $\Bbb{Z}G$ will be of the form 
		\begin{equation}\label{eqZC2}
		\delta\left(\sum_{g\in G}a_gg\right) \equiv \alpha\left(\sum_{g\in G}a_g\Big(u\tau(g)u^{-1}-\sigma(g)\Big)\right)\;mod\;[\Bbb{Q}G,\Bbb{Q}G]
		\end{equation}  
		\item There will exist central automorphisms $\sigma_1$ and $\tau_1$ of $\Bbb{Z}G$ depending on $\sigma$ and $\tau$ respectively such that $$H^1(\Bbb{Z}G,{}_{\sigma_1}\Bbb{Z}G_{\tau_1})=0.$$
	\end{enumerate}
	
\end{thm}

\begin{proof}\begin{enumerate}
		\item  Any $\Bbb{Z}$-linear $(\sigma,\tau)$-derivation $\delta$ of $\Bbb{Z}G$ can be extended to a $\Bbb{Q}$-linear $(\sigma_{\Bbb{Q}},\tau_\Bbb{Q})$-derivation of $\Bbb{Q}G$. As by hypothesis, $\sigma_\Bbb{Q},\;\tau_\Bbb{Q}$ are central $\Bbb{Q}$-automorphisms of $\Bbb{Q}G$, by Theorem \ref{previous}, part \ref{part 1}, $\delta_\Bbb{Q}$ will be $(\sigma_\Bbb{Q},\tau_\Bbb{Q})$-inner in $\Bbb{Q}G$. Hence, $\delta(g)=\alpha\tau(g)-\sigma(g)\alpha\subseteq\Bbb{Z}G$ for all $g\in G$ and some $\alpha\in\Bbb{Q}G$. Now as $G$ is finite and $\sigma$ and $\tau$ are normalized automorphisms of $\Bbb{Z}G$, we have $\tau(G)$ and $\sigma(G)$ are finite subgroups of $V(\Bbb{Z}G)$. As $\Bbb{Z}G$ satisfies $ZC2$, we have $\tau(G)$ and $\sigma(G)$ are rationally conjugate. That is, there exists  $u\in\mathcal{U}(\Bbb{Q}G)$ such that $\tau(G)=u^{-1}\sigma(G)u$. So for $g\in G$, we will have $\tau(g)=u^{-1}\sigma(h_g)u$ for some $h_g\in G.$ Then for $g\in G$, we have
	\begin{eqnarray*}\delta(g)&=&\alpha\tau(g)-\sigma(g)\alpha\\
		&=&\alpha \sigma(g)-\sigma(g)\alpha+\alpha\tau(g)-\alpha\sigma(g)\\
		&\equiv&\alpha(\tau-\sigma)(g)\;mod\; [\Bbb{Q}G,\Bbb{Q}G]\\
		&=& \alpha u^{-1}\sigma(h_g)u-\alpha\sigma(g)\;mod\;[\Bbb{Q}G,\Bbb{Q}G]\\
		&=&\alpha u^{-1}\sigma(h_g)u-\alpha u^{-1}u\sigma(h_g)+\alpha\sigma(h_g)-\alpha\sigma(g)\;mod\;[\Bbb{Q}G,\Bbb{Q}G]\\
		&\equiv&\alpha\sigma(h_g-g)\;mod\;[\Bbb{Q}G,\Bbb{Q}G]\\
		&=&\alpha \big(u\tau(g)u^{-1}-\sigma(g)\big)\;mod\;[\Bbb{Q}G,\Bbb{Q}G].
		\end{eqnarray*}
	Thus extending by linearity we get \ref{eqZC2}.\\
	 
	 \item As $\Bbb{Z}G$ satisfies $(ZC2)$, it satisfies $(AUT)$ as well (Lemma \ref{conjectures}). As $\sigma$ and $\tau$ are central automorphisms of $\Bbb{Z}G$, there exist $u,v\in\mathcal{U}(\Bbb{Q}G)$ and $\sigma_1,\tau_1\in Aut(G)$ such that $\sigma(g)=u^{-1}\sigma_1(g)u$ and $\tau(g)=v^{-1}\tau_1(g)v$ for every $g\in G.$ So $\sigma_1$ and $\tau_1$ when extended linearly will be central automorphisms of $\Bbb{Z}G$ and hence by Theorem \ref{previous} part \ref{part 2}, we will have $H^1(\Bbb{Z}G,{}_{\sigma_1}\Bbb{Z}G_{\tau_1})=0.$ 
	\end{enumerate}
	\end{proof}
\begin{rem}
	It follows from Lemma \ref{conjectures} and Theorem \ref{application} that if a group $G$ is such that $\Bbb{Z}G$ satisfies (ZC3) then for central automorphisms $\sigma$ and $\tau$ of $\Bbb{Z}G$, a $\Bbb{Z}$-linear $(\sigma,\tau)$-derivation $\delta$ of $\Bbb{Z}G$ will satisfy similar properties as Theorem \ref{application} part \ref{Appl 1}. If $\Bbb{Z}G$ satisfies (ISO) and (AUT) then in addition to the above there will exist central automorphisms $\sigma_1$ and $\tau_1$ of $\Bbb{Z}G$ depending on $\sigma$ and $\tau$ respectively such that $H^1(\Bbb{Z}G,{}_{\sigma_1}\Bbb{Z}G_{\tau_1})=0.$
\end{rem}
 Now we come to twisted derivations of commutative group algebras.

\begin{thm}\label{commutative group algebra}
	Let $G$ be an abelian group and $R$ be a commutative ring with $1$ such that either $R$ does not have torsion elements or that if $R$ has $p$-torsion elements, then $p$ does not divide the order of any element of $G$. Let $\sigma$ and $\tau$ be two different endomorphisms of $RG$. If there exists $b\in RG$ such that $\tau(b)-\sigma(b)\in\mathcal{U}(RG)$, then any $R$-linear $(\sigma,\tau)$-derivation of $RG$ is of the form: $\delta=(\tau(b)-\sigma(b))^{-1}\delta(b)(\tau-\sigma)$. 
	\end{thm} 
\begin{proof}
For any $a\in RG$, as $RG$ is abelian, we have $ab-ba=0$, where $b\in RG$ is such that $\tau(b)-\sigma(b)\in\mathcal{U}(RG).$ So,
\begin{eqnarray*}
	0&=&\delta(ab-ba)=\delta(a)\tau(b)+\sigma(a)\delta(b)-\delta(b)\tau(a)-\sigma(b)\delta(a)\\
	&=&\delta(a)\big(\tau(b)-\sigma(b)\big))+\delta(b)\big(\sigma(a)-\tau(a)\big)
	\end{eqnarray*}
Thus we have for any $a\in RG$ $$\delta(a)=\big(\tau(b)-\sigma(b)\big)^{-1}\delta(b)(\tau-\sigma)(a).$$
\end{proof}

\begin{rem}If $K$ is a field and $G$ a torsion free abelian group, we can think of $KG$ as the ring $K[x_1^{\pm1},x_2^{\pm1},\cdots]$ which is a UFD. Then by Theorem \ref{ufd thm} we can conclude that the module of all $(\sigma,\tau)$-derivations of $KG$ for two different endomorphisms $\sigma$ and $\tau$ of $KG$, will be free of rank one as a $KG$-module generated by $\delta=\frac{(\tau-\sigma)}{g}$ where $g=gcd(\sigma-\tau)(KG).$
\end{rem} 
 
 Now we give a counter example to show that Theorem \ref{previous} and thus Theorem \ref{main} does not hold in general.
 We give an example of an infinite group $G$ which is both locally finite and nilpotent and such that for every field $F$, there exists an $F$-linear $\sigma$-derivation of $FG$ which is not $\sigma$-inner.

 \begin{ex}
 
 Let $H$ be a finite non-abelian nilpotent group and ${\sigma_1}$ be an automorphism of $H$ that fixes each conjugacy class. Let $G=\Pi_{i}H_i$ be the direct product of infinitely many copies of $H$ and we call the automorphism induced by ${\sigma_1}$ to $G$ as $\sigma$, that is, $\sigma=\Pi_{i}{\sigma_1}_i$. Now, $G$ is a locally nilpotent group and can be written as $G=\bigcup_{n}G_n$ where $G_n=\Pi_{i=1}^{n}H_i.$ Consider any field $F$ and the group algebra $FG$. Let the $F$-linear extension of $\sigma$ to $FG$ be also denoted by $\sigma.$ Not that $\sigma$ will be central in $FG$. Let $\tau$ be the identity homomorphism of $FG$. Our aim is to find an $F$-linear $\sigma$-derivation of $FG$ that is not $\sigma$-inner on $FG$.\\
  Let $x_i$ be any non-central element of $H_i$ and define $\delta: FG\longrightarrow FG$ by $$\delta(\alpha)=\sum_{i=1}^{\infty}x_i\alpha-\sigma(\alpha)x_i.$$ Notice that, if $\alpha\in FG_n$, then all summands after the $n$-th term are zero and hence $\delta(\alpha)\in FG_n.$  Now $\delta$ is a $\sigma$-derivation of $FG$ with $\delta$ restricted to $FG_n$ being the $\sigma$-inner derivation induced by $x_1+x_2+\cdots+x_n$. Now, $\delta(H_i)\neq 0$ for all indices $i$. Thus there exists no element $x$ with finite support such that $\delta(\alpha)=x\alpha-\sigma(\alpha)x$. That is, there does not exist any element $x$ in the group algebra $FG$ such that $\delta$ is $\sigma$-inner induced by $x$. Thus $\delta$ is a $F$-linear $\sigma$-derivation of $FG$, but it is not $\sigma$-inner.

\end{ex}\label{Counter Example}



\bibliographystyle{alpha}
\bibliography{OneBibToRuleThemAll}

\end{document}